\documentclass[reqno, 11pt]{amsart}%
\usepackage{amssymb}
\usepackage[nesting]{hyperref}
\usepackage[pdftex]{graphicx}
\usepackage{listings}
\usepackage{multirow}
\usepackage{placeins}
\usepackage{color}
\usepackage{subfigure}
\usepackage{lscape}
\usepackage{dsfont}
\usepackage{amsmath}
\usepackage{amsfonts}
\usepackage{tikz}
\usepackage{verbatim}
\usepackage{accents} 
\usepackage{todonotes}
\usepackage{mathtools}
\usepackage{bbm}

\newcommand{\BlackBoxes}{\global\overfullrule5pt}

\BlackBoxes

\newcommand{\R}{\mathbb{R}} 
\newcommand{\Q}{\mathbb{Q}} 
\newcommand{\dx}{\mathrm{d}}
\newcommand{\PP}{\mathbb{P}}
\newcommand{\EE}{\mathbb{E}}
\newcommand{\E}{\mathbb{E}}

\newcommand{\st}{\text{s.t. }}

\newcommand{\AMF}{\mathcal{A}^{\text{MF}}}

\mathtoolsset{showonlyrefs}

\newtheorem{theorem}{Theorem}

\theoremstyle{definition}
\newtheorem{example}[theorem]{Example}
\newtheorem{remark}[theorem]{Remark}
\newtheorem{definition}[theorem]{Definition}

\numberwithin{equation}{section} \numberwithin{theorem}{section}
\def\0{\kern0pt\-\nobreak\hskip0pt\relax}

\makeatletter
\AtBeginDocument{ \def\@serieslogo{ \vbox to\headheight{ \parindent\z@ \fontsize{6}{7\p@}\selectfont
\vss}}}

\def\makeoverbar#1#2#3#4#5#6#7{ \setbox0=\hbox{$\m@th#2\mkern#5mu{{}#3{}}\mkern#6mu$} \setbox1=\null \dimen@=#4\fontdimen8#13 \dimen@=3.5\dimen@
\advance\dimen@ by \ht0 \dimen@=-#7\dimen@ \advance\dimen@ by \wd0
\ht1=\ht0 \dp1=\dp0 \wd1=\dimen@
\dimen@=\fontdimen8#13 \fontdimen8#13=#4\fontdimen8#13
\rlap{\hbox to \wd0{$\m@th\hss#2{\overline{\box1}}\mkern#5mu$}}
\fontdimen8#13=\dimen@}
\def\mylabel#1#2{{\def\@currentlabel{#2}\label{#1}}}
\makeatother
\begin{document}
\title[  ]{Nash equilibria for relative investors via no-arbitrage arguments}
\author[N. \smash{B\"auerle}]{Nicole B\"auerle}
\address[N. B\"auerle]{Department of Mathematics,
Karlsruhe Institute of Technology (KIT), D-76128 Karlsruhe, Germany}

\email{\href{mailto:nicole.baeuerle@kit.edu}{nicole.baeuerle@kit.edu}}

\author[T. \smash{G\"oll}]{Tamara G\"oll}
\address[T. G\"oll]{Department of Mathematics,
Karlsruhe Institute of Technology (KIT), D-76128 Karls\-ruhe, Germany}

\email{\href{mailto:tamara.goell@kit.edu} {tamara.goell@kit.edu}}


\begin{abstract}
Within a common arbitrage-free semimartingale financial market we consider the problem of determining all Nash equilibrium investment strategies for $n$ agents who try to maximize the expected utility of their relative wealth. The utility function can be rather general here. Exploiting the linearity of the stochastic integral and making use of the classical pricing theory we are able to express all Nash equilibrium investment strategies in terms of the optimal strategies for the classical one agent expected utility problems. The corresponding mean field problem is solved in the same way. We give four applications of specific financial markets and compare our results with those given in the literature.
\end{abstract}
\maketitle


\makeatletter \providecommand\@dotsep{5} \makeatother



\vspace{0.5cm}
\begin{minipage}{14cm}
{\small
\begin{description}
\item[\rm \textsc{ Key words}]
{\small Portfolio optimization; Semimartingale market; Nash equilibrium; Relative investor}
\item[\rm \textsc{AMS subject classifications}] 
{\small 91A35, 91A16, 91G20}

\end{description}
}
\end{minipage}

\section{Introduction}
In this paper we consider an $n$-player investment problem where the individual agents try to maximize their expected utility from relative wealth measured against the performance of the other agents. Problems like this have been motivated for example in \cite{brown2001careers,kempf2008tournaments} by competition between agents. In \cite{basak2015competition} a continuous-time model for two agents is considered with stocks following a geometric Brownian motion.  Power utilities are used to evaluate the relative investment performance where relative wealth is interpreted as the ratio of the wealth of the two agents. The authors derive the unique pure-strategy Nash equilibrium. \cite{espinosa2015optimal} also use stocks following a geometric Brownian motion and model the relative concerns of $n$ agents using the arithmetic mean of the wealth of all $n$ agents. They derive existence and uniqueness results for Nash equilibria under additional constraints on the admissible strategies.  

The problem has been taken up and extended in \cite{lacker_zariphopoulou_n-agent_nash}. There $n$ agents are considered who can invest into their own financial markets which are correlated Black-Scholes markets. In case of a power utility (CRRA) relative wealth is again measured as the ratio of own wealth against weighted competitors' wealth, where in case of an exponential utility (CARA) relative wealth is measured as the difference between own wealth and weighted competitors' wealth. A Nash equilibrium in constant strategies is derived in the different settings and finally the mean field problem is considered where the number of agents tends to infinity. 

Later, a number of papers appeared which generalized the problem in various directions. For example \cite{lacker2020many} consider the problem with consumption-investment where both terminal wealth and consumption are measured by relative wealth for  CRRA utilities. In \cite{dos2019forward,reis2020forward} the process of forward utilities of the finite player game and the forward utilities of the mean field game are considered with and without consumption. Forward utilities are of CARA and CRRA type, respectively. Generalizations of the financial market can be found in \cite{kraft2020dynamic,fu2020mean,hu2021n}. The first paper \cite{kraft2020dynamic} considers a stochastic volatility model (CIR) and CRRA utility for two players. \cite{fu2020mean} treat a financial market based on Brownian motions with stochastic coefficients. They consider a CARA utility and Nash equilibria are characterized as solutions of FBSDEs. Finally, \cite{hu2021n} deal with a common It\^{o}-diffusion market for all agents which may be incomplete in case of CARA utilities or complete in case of random risk tolerance coefficients. Further papers among others are \cite{deng2020relative} where the problem with partial information and heterogeneous priors is considered and \cite{whitmeyer2019relative} which discusses more economical questions like the structure of equilibria and the effect of additional agents.

In this paper we start with a very general arbitrage-free semimartingale financial market which comprises all common market models treated so far and extends the literature in this area by allowing for jumps in the stock price processes. It is important that we assume that every agent has the same financial market as investment opportunity and the relative wealth is measured by the difference between own wealth and  weighted sum of competitors' wealth. We make almost no assumption about the utility function. In particular we are also able to consider e.g.\ Cumulative Prospect Theory (CPT) which has not been done before. Using classical pricing theory and exploiting the fact that the stochastic integral is linear in the trading strategy we are able to characterize {\em all } Nash equilibria in our model in terms of the optimal strategy for the classical one agent utility maximization problem. More precisely, once the classical problem has been solved we derive the Nash equilibria as the unique solution of a system of  linear equations. This is because we can split the initial wealth of each agent into a part which is used to hedge the competitors' wealth and a part which is used for own optimization. Uniqueness is an issue which has not been solved in all of the previous literature and posed as an open question in \cite{lacker_zariphopoulou_n-agent_nash}. Similar structural arguments can be found in \cite{bauerle2015complete}. The idea then carries over to the mean field problem. Using our method of proof it is also easy to see that the Nash equilibrium strategies coincide with Pareto-optimal strategies. 

The outline of our paper is as follows: In the next section we introduce the financial market and the optimization problem. In Section \ref{sec:Solution} we explain our solution approach and state the first main theorem about existence, uniqueness and form of the Nash equilibrium strategies. In Section \ref{sec:examples} we give four examples where we compute Nash equilibrium strategies in different settings. We present the solution in the classical terminal utility maximization setting in a general Lévy market, the Heston model and the Cox-Ross-Rubinstein market in discrete time and also consider the case of CPT utility, where the underlying financial market is chosen to be a Black-Scholes market with constant parameters. In the final section we motivate and solve the corresponding mean field problem. Besides some technical obstacles we essentially get the same results.

\section{Financial Market Model and Problem}\label{sec:market}
We consider the following general underlying financial market taken from \cite{vcerny2007structure,delbaen1996attainable}:

Let $(\Omega, \mathcal{F}, (\mathcal{F}_t), \mathbb{P})$ be a filtered probability space and $T>0$ be a finite time horizon. The underlying financial market consists of $d$ stocks and one riskless bond, each defined on the previously mentioned probability space. The price process of the $d$ stocks will be a non-negative c\`{a}dl\`{a}g  $L^2(\mathbb{P})$-semimartingale $S=(S_1(t),\ldots,S_d(t))_{t\in [0,T]}$, i.e. we assume that
\begin{equation}
\sup\{ \EE [(S_k(\tau))^2] : \tau \mbox{ is a }  (\mathcal{F}_t)-\mbox{stopping time, } k=1,\ldots,d \}<\infty.
\end{equation}
The price process  of the riskless bond is for simplicity assumed to be identical to $1$. We make the following assumption \cite{vcerny2007structure}:\\

{\em Assumption:}
There exists a probability measure $\mathbb{Q} \sim \mathbb{P}$ such that $\EE \left( \frac{\dx \mathbb{Q}}{\dx\mathbb{P}}\right)^2< \infty$ and $S$ is a $\mathbb{Q}$-$\sigma$-martingale. Such a $\mathbb{Q}$ is called $\sigma$-martingale measure (S$\sigma$MM) with square integrable density.
\\

Note that this assumption is equivalent to ruling out   free-lunch  \cite{vcerny2007structure}. In a next step we introduce a suitable set of trading strategies. Since the riskless bond is equal to 1 and we only consider self-financing strategies we can express the wealth process with the help of the investment in risky assets only. By $\varphi = (\varphi_1,\ldots,\varphi_d)$ we denote a $d$-dimensional stochastic process representing the trading strategy of some investor, where $\varphi_k(t)$ describes the number of shares of the $k$-th stock held at time $t\in [0,T]$.  We assume that $\varphi\in L(S)$, where $L(S)$ denotes the set of $(\mathcal{F}_t)$-predictable, $S$-integrable stochastic processes. The associated wealth process denoted by $(X^\varphi_t)_{t\in [0,T]}$ is given by 
\begin{equation} 
X^\varphi_t = x_0+ (\varphi \cdot S)_t, \mbox{ where } (\varphi \cdot S)_t=\sum_{k=1}^d \int_0^t\varphi_k(u) \mathrm{d}S_k(u) \label{eq sde price process}
\end{equation}
with initial  capital $x_0 \in \R$. 
A trading strategy $\varphi$ is called \textit{admissible}, if  it is in 
\begin{align*}\nonumber
 \mathcal{A} \coloneqq  &\, \Big\{\varphi \in L(S)\, |\, (\varphi\cdot S)_T \in L^2(\mathbb{P}), \; (\varphi \cdot S) Z^\mathbb{Q} \text{ is a $\mathbb{P}$-martingale for all S}\sigma\text{MM } \mathbb{Q} \\
 & \; \text{ with density process } Z^\mathbb{Q} \text{ and  } \frac{\dx \mathbb{Q}}{\dx \mathbb{P}} \in L^2(\mathbb{P})\Big\}.
\end{align*}



In the previously described setting, the time-zero price of any claim $X\in L^2(\mathbb{P})$ is given by 
\begin{equation}
 \mathbb{E}_{\mathbb{Q}}\left[{X} \right]  = \EE\left[Z^\mathbb{Q}_T X \right],
\end{equation}
for all S$\sigma$MM $\mathbb{Q}$ with $\frac{\dx \mathbb{Q}}{\dx \mathbb{P}} \in L^2(\mathbb{P})$. In particular for $X_T^\varphi$ from \eqref{eq sde price process} we obtain                  $\mathbb{E}_{\mathbb{Q}}\left[X_T^\varphi\right] = x_0.$   \\

Further we assume that there are $n$ investors trading in the previously described financial market. Each investor measures her preferences with respect to some  utility function $U_i:\mathcal{D}_i \rightarrow \R$ applied to the difference of her own terminal wealth and a weighted arithmetic mean of the other $n-1$ agents' wealth. This criterion has been used in \cite{lacker_zariphopoulou_n-agent_nash}. A utility function is here defined as follows.

\begin{definition}
Let $\mathcal{D}\in \{(0,\infty), [0,\infty), \R\}$. A function $U:\mathcal{D}\rightarrow \R$ is called \textit{ utility function} if $U$ is continuous and strictly increasing. 
\end{definition}

For convenience we extend all utility functions to functions on $\R$ by setting $U(x)=-\infty$ if $x\notin \mathcal{D}$. For given \textit{competition weights} $\theta_i \in [0,1]$ and initial capital $x_0^i \in \mathcal{D}_i$, investor $i$ aims to solve 
\begin{equation}
\begin{cases}
& \sup_{\varphi^i \in \mathcal{A}} \mathbb{E}\left[U_i \left( X_T^{i,\varphi^i} - \theta_i \bar{X}^{-i,\varphi}_T \right) \right],\\
\st & X^{i,\varphi^i}_T = x_0^i + (\varphi^i \cdot S)_T,
\end{cases}\quad i=1,\ldots,n, \label{eq optimization problem arithm}
\end{equation}
where $\bar{X}^{-i,\varphi}_T = \frac{1}{n}\sum_{j\neq i} X_T^{j,\varphi^j}$ and $\varphi^j, j\neq i,$ are fixed admissible strategies. For the later analysis it is a little bit more convenient to scale the sum by $n$ instead of $n-1$.

\begin{remark}
Note that we need to ensure that the investor is able to attain a terminal wealth $ X_T^{i,\varphi^i}$ such that $ X_T^{i,\varphi^i} - \theta_i \bar{X}^{-i,\varphi}_T \in \mathcal{D}_i$  $\mathbb{P}$-almost surely. Otherwise the problem is trivial. We will later see that this condition is satisfied when we choose the competition weight $\theta_i \in [0,1]$ under the constraint $x_0^i -  \frac{\theta_i}{n} \sum_{j \neq i}x_0^j \in \mathcal{D}_i$. Obviously, this constraint is only relevant if $\mathcal{D}_i \in \{ (0,\infty),[0,\infty)\}$. In this case, we need to make sure that $$x_0^i -  \frac{\theta_i}{n} \sum_{j \neq i}x_0^j >0\ (\geq 0)$$ is satisfied. If $\mathcal{D}_i = (0,\infty)$, this is equivalent to 
\begin{align*}
\left(1+\frac{\theta_i}{n} \right)x_0^i - \frac{\theta_i}{n}\sum_{j=1}^n x_0^j > 0  \Longleftrightarrow 
\frac{\frac{\theta_i}{n}}{1+\frac{\theta_i}{n}} < \frac{x_0^i}{\sum_{j=1}^n x_0^j} \Longleftrightarrow \frac{\theta_i}{n} < \frac{\alpha_i}{1-\alpha_i},
\end{align*}
where \begin{equation*}
\alpha_i := \frac{x_0^i}{\sum_{j=1}^n x_0^j}
\end{equation*}
describes the fraction of the collective initial capital originating from investor $i$. The case $\mathcal{D}_i = [0,\infty)$ follows analogously. Therefore the constraint $x_0^i -  \frac{\theta_i}{n} \sum_{j \neq i}x_0^j >0$ implies an upper bound on the choice of $\theta_i$, which is obviously increasing in the fraction $\alpha_i$.
Hence we obtain an upper bound on the competition weight which is increasing in terms of the $i$-th agent's initial capital and decreasing in terms of the other $n-1$ agents' initial capital. We can interpret this observation as follows: The more an investor contributes in the beginning, the more she may care about the other agents' investment behavior.
\end{remark} 

\begin{remark}\label{remark:discrete_vs_continuous}
The arguments presented in Section \ref{sec:Solution} also apply for financial markets in discrete time. In order to keep the setting as simple but also as general as possible, we decided to introduce a setting in continuous time that covers many important market models (see Section~\ref{sec:examples} for some examples). However, the arguments in the following Section do not depend at all on the underlying financial market, so that one could also have a general arbitrage-free financial market in discrete time in mind, e.g. \cite{pliska2005introduction}, Chapter 3. We explain this later in Example \ref{subsec:CRR}. 
\end{remark}

\section{Solution via Pricing Theory}\label{sec:Solution}
 
Next we explain how to solve the multi-objective optimization problem given by \eqref{eq optimization problem arithm} in the context of Nash equilibria. A Nash equilibrium for agents using general objective functions $J_i$ is defined as follows.

\begin{definition}\label{definition nash}
Let $J_i:\mathcal{A}^n \rightarrow \R$ be the objective function of agent $i$. A vector $\left(\varphi^{1,*},\ldots,\varphi^{n,*}\right)$ of strategies is called a \emph{Nash equilibrium}, if, for all admissible $\varphi^i \in \mathcal{A}$ and $i=1,\ldots,n$, 
\begin{equation}
J_i(\varphi^{1,*},\ldots,\varphi^{i,*},\ldots,\varphi^{n,*})\geq J_i(\varphi^{1,*},\ldots,\varphi^{i-1,*},\varphi^i,\varphi^{i+1,*},\ldots,\varphi^{n,*}). \label{Nash condition}
\end{equation}
\end{definition}

In the context of Nash equilibria each investor tries to maximize her own objective function while the strategies of the other investors are assumed to be fixed. This maximization results in the optimal strategy of agent $i$ in terms of the strategies of the other investors. 
The second step of the optimization process is a fixed-point problem in order to find the $n$-tuple of admissible strategies to satisfy each investors' preference determined in the first optimization step. 

Generally, one would proceed in the previously described way by fixing some investor $i$, fixing the other agents' strategies, maximizing the $i$-th objective function and solving the fixed-point problem afterwards. We choose a different approach to find a Nash equilibrium and discuss its uniqueness. 

At first, we consider the expression inside the utility function in \eqref{eq optimization problem arithm}. Since agent $i$ can only control her own strategy, the random variable $\theta_i \bar{X}^{-i,\varphi}_T$ can be understood as some fixed asset claim. An arbitrary strategy $\varphi$ can due to linearity always be decomposed into 
$$ X_T^\varphi = X_T^\varphi - X_T^{\varphi'} + X_T^{\varphi'} = X_T^{\varphi-\varphi'} + X_T^{\varphi'},  $$ 
for arbitrary $\varphi'$. Investor $i$ can without loss of generality invest some fraction of her initial capital in order to hedge the claim $\theta_i \bar{X}^{-i,\varphi}_T$. The remaining part of her initial capital can then be used to maximize her own terminal wealth. This idea leads to a much simpler method of determining Nash equilibria in the given context.

The first step is to determine the price of the claim $\theta_i \bar{X}^{-i,\varphi}_T$. Each investor $j$, where $j\neq i$, has some initial capital $x_0^j$. Hence the time zero price of $X_T^{j,\varphi^j}$ equals the initial capital $x_0^j$. By linearity, the price of the claim $\theta_i \bar{X}^{-i,\varphi}_T$ is simply given by 
$$ \theta_i\bar{x}_0^{-i}:= \frac{\theta_i}{n} \sum_{j\neq i } x_0^j.$$
Note that the time zero price is independent of the strategies $\varphi^j$, $j\neq i$, chosen by the other investors. Hence the maximization problem in the second step does not depend on the other $n-1$ agents' strategies.

In the second step, investor $i$ needs to solve a classical portfolio optimization problem using the utility function $U_i$ and the reduced initial capital $\tilde{x}_0^i := x_0^i - \theta_i \bar{x}_0^{-i}$. The portfolio optimization problem
\begin{equation}
\begin{cases}
&\sup\limits_{\psi^i \in \mathcal{A}} \mathbb{E}\left[U_i\left(Y^{i,\psi^i}_T\right) \right],\\
\text{s.t. } & Y_T^{i,\psi^i} = \tilde{x}_0^i+ (\psi^i \cdot  S)_T,\label{eq auxiliary problem nash}
\end{cases}
\end{equation}
can be solved using standard methods. We assume here that there exists an optimal strategy  $\psi^{i,*}\in \mathcal{A}$ for \eqref{eq auxiliary problem nash}. 

In the last step the Nash equilibria are determined using the linearity of the price process. By construction the process $(Y_t^{i,\psi^{i,*}})_{t\in [0,T]}$ can be written as 
\begin{equation}
Y_t^{i,\psi^{i,*}} = X_t^{i,\varphi^i} - \frac{\theta_i}{n} \sum_{j\neq i} X_t^{j,\varphi^j}\label{eq hedging variable}
\end{equation}
$\mathbb{P}$-almost surely for all $t\in [0,T]$. Then, using the linearity of the wealth (of course we assume that stock prices are not identical), the strategies  $\varphi^j \in \mathcal{A}$ are obtained from 
\begin{equation}
\psi_k^{i,*}(t) = \varphi_k^i(t) - \frac{\theta_i}{n} \sum_{j\neq i} \varphi^j_k(t),\quad  i=1,\ldots,n,\,  k=1,\ldots,d\label{eq sle}
\end{equation}
$\mathbb{P}$-almost surely for all $t\in [0,T]$. Hence the Nash equilibria can be determined as the solution to the system of linear equations defined by \eqref{eq sle} where $\psi_k^{i,*}$ are given. If \eqref{eq auxiliary problem nash} and \eqref{eq sle}  have a unique solution, the resulting Nash equilibrium is unique as well. The question of uniqueness has been posed as an open problem in \cite{lacker_zariphopoulou_n-agent_nash}. Our setting gives a partial answer to their question.

\begin{theorem}\label{thm:unique_NE}
If \eqref{eq auxiliary problem nash} has a unique (up to modifications) optimal portfolio strategy $\psi^{i,*}$ for all $i$, then there exists a unique Nash equilibrium for \eqref{eq optimization problem arithm} given by
\begin{equation}
\varphi_k^{i}(t) := \frac{n}{n+\theta_i} \psi_k^{i,*}(t) + \frac{\theta_i}{(n+\theta_i)( 1-\hat{\theta})}\sum_{j=1}^n \frac{n}{n+\theta_j}\psi_k^{j,*}(t),\quad i=1,\ldots,n, k=1,\ldots,d \label{eq nash equilibrium}
\end{equation}
$\mathbb{P}$-almost surely for all $t\in [0,T]$, where
\begin{equation}
\hat{\theta} := \sum_{i=1}^n \frac{\theta_i}{n+\theta_i}.
\end{equation}
\end{theorem}

\begin{proof}
Independent of the exact choice of $\psi^{i,*}$, we can determine the solution to \eqref{eq sle} in terms of $\psi^{i,*}$ for all $t \in [0,T]$ and $\mathbb{P}$-almost all $\omega \in \Omega$. Therefore we fix some arbitrary $t\in [0,T]$ and $\omega \in \Omega$ (we omit the argument $\omega$ in the following calculations) and define (componentwise)
\begin{equation}
\hat{\varphi}(t) :=\sum_{j=1}^n \varphi^j(t).\label{eq pi_hat def}
\end{equation}
Hence \eqref{eq sle} is equivalent to 
\begin{align*}
\psi_k^{i,*}(t) = \frac{n+\theta_i}{n}\varphi_k^i(t) - \frac{\theta_i}{n} \hat{\varphi}_k(t)
\end{align*}
and therefore $\varphi^i$ is implicitly given by 
\begin{equation}
\varphi_k^i(t) = \frac{n}{n+\theta_i}\left(\psi_k^{i,*}(t) + \frac{\theta_i}{n} \hat{\varphi}_k(t) \right). \label{eq pi_i implicit}
\end{equation}
Inserting \eqref{eq pi_i implicit} into \eqref{eq pi_hat def} yields
\begin{align*}
\hat{\varphi}_k(t) = \sum_{i=1}^n \varphi_k^i(t) &= \sum_{i=1}^n \frac{n}{n+\theta_i}\psi_k^{i,*}(t) + \hat{\varphi}_k(t)\sum_{i=1}^n \frac{\theta_i}{n+\theta_i}\\
&= \sum_{i=1}^n  \frac{n}{n+\theta_i}\psi_k^{i,*}(t) + \hat{\theta} \hat{\varphi}_k(t)\\
\Leftrightarrow \left(1-\hat{\theta} \right) \hat{\varphi}_k(t) &= \sum_{i=1}^n  \frac{n}{n+\theta_i}\psi_k^{i,*}(t),
\end{align*}
where we used the abbreviation of $\hat{\theta}$.
Since $\theta_i \in [0,1]$, $i=1,\ldots,n$, we obtain 
\begin{align*}
\hat{\theta} = \sum_{i=1}^n \frac{\theta_i}{n+\theta_i} = \sum_{i=1}^n \left(1 + \frac{n}{\theta_i}\right)^{-1} \leq \sum_{i=1}^n \left(1 + n\right)^{-1} = \frac{n}{n+1} < 1.
\end{align*}
Therefore we can deduce further that
\begin{equation}
\hat{\varphi}_k(t) = \frac{1}{1-\hat{\theta}} \sum_{i=1}^n \frac{n}{n+\theta_i}\psi_k^{i,*}(t) \label{eq pi_hat rep}
\end{equation}
holds $\mathbb{P}$-almost surely for all $t\in [0,T]$. Finally, inserting \eqref{eq pi_hat rep} into \eqref{eq pi_i implicit} yields  the stated expression for the Nash equilibrium. The  line of arguments implies that there exists a unique Nash equilibrium given by \eqref{eq nash equilibrium} if and only if there exists a unique optimal portfolio strategy $\psi^{i,*}$ to the auxiliary problem \eqref{eq auxiliary problem nash}. 
\end{proof}

\begin{remark}
Suppose that instead of the Nash equilibrium we want to determine a Pareto optimal strategy of the system, i.e. we maximize
$$\sum_{i=1}^{n}\beta_i \sup_{\varphi^i \in \mathcal{A}} \mathbb{E}\left[U_i \left( X_T^{i,\varphi^i} - \theta_i \bar{X}^{-i,\varphi}_T \right) \right],\quad \beta_i >0,\, i=1,\ldots,n,$$
the weighted sum of all individuals' relative wealth over all strategies. Following the same line of arguments as before, we can see that optimal strategies are the same as for the Nash equilibrium. This means the social and the individual optimal strategies are the same.
\end{remark}

\begin{remark}
The described method is not limited to classical expected utility maximization. Some examples of other types of optimization problems that can be treated with the described method are the VaR-based optimization by Basak and Shapiro \cite{basak2001value}, the rank-dependent utility model with a VaR-based constraint by He and Zhou \cite{he2016hope} or general mean-variance problems that can be found for example in \cite{korn1997optimal}. Another example is the Cumulative Prospect Theory (CPT) by Tversky and Kahneman \cite{tversky1992advances} that is further analyzed in Subsection \ref{subsec:cpt}.
\end{remark}

\begin{remark}\label{remark:mon}
Since the Nash equilibrium given in Theorem \ref{thm:unique_NE} explicitly contains the optimal solutions to the associated classical portfolio optimization problems (if we set $\theta_i=0$, the agent does not care about the other agents and just solves the classical portfolio problem, i.e. $\varphi_k^i=\psi_k^{i,*}$), we can compare the optimal solution $\psi_k^{i,*}$ of the classical problem to the associated component $\varphi_k^i$ of the Nash equilibrium. If $\psi_k^{i,*}> 0 \, (<0)$ and $\sum_{j\neq i} \frac{n}{n+\theta_j} \psi_k^{j,*}>0 \, (<0)$, $\varphi_k^i(t)$ is increasing (decreasing) in terms of $\theta_i$, which can be seen directly considering the partial derivative
\begin{align*}
    \frac{\partial}{\partial \theta_i} \varphi_k^i(t) =& \frac{n}{(n+\theta_i)^2} \left(1+\frac{\theta_i}{(n+\theta_i)(1-\hat{\theta})} \right)\left(\frac{n}{(n+\theta_i)(1-\hat{\theta})}-1 \right)\psi_k^{i,*}(t)\\
    & + \frac{n}{(n+\theta_i)^2(1-\hat{\theta})}\left(1+\frac{\theta_i}{(n+\theta_i)(1-\hat{\theta})} \right)\sum_{j\neq i} \frac{n}{n+\theta_j}\psi_k^{j,*}(t).
\end{align*}
Hence, under these conditions in a competitive environment ($\theta_i>0$) agent $i$ invests more in the stocks than in a classical non-competitive environment ($\theta_i=0$).
\end{remark}

\section{Examples} \label{sec:examples}
In this section we consider in Subsections \ref{subsec:levy} - \ref{subsec:cpt} special cases of our general market model. In Subsection \ref{subsec:CRR} we show that the reasoning also works for a discrete-time financial market.
\subsection{L\'evy market}\label{subsec:levy}
Our first example is a L\'evy market consisting of a riskless bond with interest rate $r=0$ and $d$ stocks. Thus, let $W$ be a $d$-dimensional Brownian motion and $N$ a Poisson random measure on $(-1,\infty)^d$, i.e.\ $N([0,t]\times B)$ for
$t\geq 0$ and $B\in\mathcal{B}((-1,\infty)^d)$ is the number of all
jumps taking values in the set $B$ up to the time $t$. The
associated L\'{e}vy measure is denoted by $\nu$, i.e. $\nu(B) = \mathbb{E} N([0,1]\times B)$ gives
the expected number of jumps per unit time whose size belong to $B$.
For notational convenience let us define
$\overline{N}(\mathrm{d}t,\mathrm{d}z):=N(\mathrm{d}t,\mathrm{d}z)-\mathds{1}_{\{|\!|z|\!|<1\}}\nu(\mathrm{d}z)\mathrm{d}t$.
The price processes for  $k=1,\ldots,d$ are  given by
\begin{align*}
\mathrm{d}S_k(t)=&S_k(t_-)\Big(\mu_k\,
\mathrm{d}t+\sum_{\ell =1}^d\sigma_{k\ell }\,\mathrm{d}W_{\ell }(t)+\int_{(-1,\infty)^d}z_k\,\overline{N}(\mathrm{d}t,\mathrm{d}z)\Big)
\end{align*}
where $S_k(0)=1$, $\sigma_{k\ell}\ge 0$. By $\mu = (\mu_1,\ldots,\mu_d)$ we denote the drift vector and by $\sigma = (\sigma_{k\ell})_{1\leq k,\ell \leq d}$ the volatility matrix, which is assumed to be regular.
Moreover there are $n$ investors trading in the L\'evy financial market with initial capital $x_0^i$, $i=1,\ldots,n$.  

Assume that each investor uses an exponential utility function 
\begin{equation}
U_i:\R \rightarrow \R,\ U_i(x) = -\exp\left(-\frac{1}{\delta_i}x\right)
\end{equation}
for some parameter $\delta_i > 0$, $i=1,\ldots,n$. Hence the associated objective function is given by 
\begin{equation}
\mathbb{E}\left[-\exp\left(-\frac{1}{\delta_i}\left( X_T^{i,\varphi^i} - \theta_i \bar{X}^{-i,\varphi}_T\right) \right) \right]\label{eq objective exponential}
\end{equation}
for the competition weight parameters $\theta_i\in [0,1]$, $i=1,\ldots,n$. 
We assume that the market is free of arbitrage for an appropriate strategy class and that
\begin{align}\label{integrability_cond1}
 \int\limits_{|\!|z|\!|>
1}\!\!|\!|z|\!|\exp\left(\frac{1}{\delta_i}\Lambda_i
|\!|z|\!|\right)\,\nu(\mathrm{d}z)<\infty
\end{align}
for constants $0<\Lambda_i<\infty$, $i=1,\ldots,n$.
Now we use the  described method  to determine the Nash equilibrium in the given situation.
First, the unique (up to modifications) optimal solution to the optimization problem 
\begin{equation}\label{eq optimization problem auxiliary}
\begin{cases}
& \sup_{\psi^i\in \mathcal{A}} \mathbb{E}\left[-\exp\left(-\delta_i^{-1} Y_T^{i,\psi^i} \right) \right],\\
\text{s.t.} & Y_T^{i,\psi^i} = \tilde{x}_0^i+ (\psi^i \cdot  S)_T,
\end{cases}
\end{equation}
for $\tilde{x}_0^i = x_0^i - \frac{\theta_i}{n}\sum_{j\neq i} x_0^j$ is given by
\begin{equation}
\psi_k^{i,*} (t) = \frac{ \pi_k^{i,*}}{S_k(t)}, \label{eq candidate}
\end{equation}
where  $(\pi_1^{i,*},\ldots,\pi_d^{i,*})$ is the solution of
\begin{equation}\label{eq:dzero}
0=\mu_k -\frac{1}{\delta_i}\sum_{\ell=1}^d\sum_{r=1}^{d}
\pi_r^{i,*}\sigma_{k\ell}\sigma_{r \ell}+ \int_{(-1,\infty)^d} z_k
\left(\exp\left(-\frac{1}{\delta_i} \sum_{r=1}^d\pi_r^{i,*}z_r\right)
-\mathds{1}_{\{|\!|z|\!|<1\}}\right)\,\nu(\mathrm{d}z),\, 
\end{equation}
$k=1,\ldots,d$, which we assume to exist and be unique, \cite{bauerle2011optimal}.
Since $\hat{\theta} < 1$, we know that the unique Nash equilibrium is then given by 
\begin{equation}
\varphi_k^{i,*}(t) = \frac{n}{n+\theta_i} \psi_k^{i,*}(t) + \frac{\theta_i}{(n+\theta_i)(1-\hat{\theta} )}\sum_{j=1}^n \frac{n}{n+\theta_j} \psi_k^{j,*}(t) \label{eq Nash second step}
\end{equation}
$\mathbb{P}$-almost surely for all $t\in [0,T]$ and hence the Nash equilibrium in terms of invested amounts $\pi_k^i(t) := \varphi_k^{i,*}(t) S_k(t)$ is deterministic and constant. 

\begin{remark}\label{remark:no_jumps}
a) \, If we skip the jump component, the Lévy market reduces to a $d$-dimensional Black-Scholes market with constant parameters. In this case, the unique optimal invested amount in the single investor problem is known to be given by 
\begin{equation}
\widetilde{\pi}^{i,*} \coloneqq  \delta_i (\sigma \sigma^\top)^{-1} \mu \label{eq single agent exponential}
\end{equation} 
and hence we obtain the unique Nash equilibrium 
\begin{equation}
\pi^{i,*} = \underbrace{\left(\frac{n\delta_i}{n+\theta_i} + \frac{\theta_i}{(n+\theta_i)(1-\hat{\theta} )} \sum_{j=1}^n \frac{n\delta_j}{n+\theta_j} \right)}_{\eqqcolon C_i}\cdot \left(\sigma \sigma^\top \right)^{-1}\mu, \ i=1,\ldots,n.\label{eq NE BS exp}
\end{equation}

A comparison of the optimal portfolio \eqref{eq single agent exponential} in the single agent problem and the Nash equilibrium given by \eqref{eq NE BS exp} shows that the overall structure is the same. In both cases, the optimal amount invested into the $d$ stocks is constant over time and deterministic. Moreover the optimal investment is given as $(\sigma \sigma^\top)^{-1} \mu$ multiplied by some constant factor. It can be shown that the constant $C_i$ in the Nash equilibrium is strictly increasing in terms of $\theta_i$ (see also Remark \ref{remark:mon}). 
By inserting the  optimal investment strategy we can deduce that the optimal terminal wealth of agent $i$ is given by 
\begin{align*}
X_i^* &= x_0^i  + C_i \Vert \sigma^{-1} \mu \Vert^2 T + C_i (\sigma^{-1}\mu)^\top W_T.  
\end{align*}
Then the expected terminal wealth is obviously 
\begin{equation}
\mathbb{E}[X_i^*] = x_0^i + C_i \Vert \sigma^{-1} \mu \Vert^2 T 
\end{equation}
and therefore strictly increasing in terms of $\theta_i$. Hence, in order to maximize the expected terminal wealth, investor $i$ should choose the competition weight $\theta_i = 1$. 

However, choosing a high competition weight also brings the disadvantage of increasing the probability of a loss (with respect to the initial capital $x_0^i$). To see this, we choose some constant $K < x_0^i$ and consider the probability that the optimal terminal wealth $X_i^*$ is less or equal than $K$. It follows
\begin{align*}
\mathbb{P}(X_i^* \leq K) &=  \Phi \left(\frac{K- x_0^i}{C_i \Vert \sigma^{-1} \mu \Vert \sqrt{T}} - \Vert \sigma^{-1} \mu \Vert \sqrt{T}\right),
\end{align*}
where $\Phi$ denotes the distribution function of the standard normal distribution. The expression inside $\Phi$ is increasing in $\theta_i$ since $K - x_0^i < 0$. Hence the probability that the terminal wealth is significantly smaller than the initial wealth is strictly increasing in terms of $\theta_i$.

b)\,  If we further set $d=1$, $\mu_1 = \mu$, $\sigma_1 = \sigma>0$ , we obtain the constant Nash equilibrium in Corollary 2.4 in \cite{lacker_zariphopoulou_n-agent_nash}. The authors there use the slightly different objective function 
\begin{equation}
\mathbb{E}\left[-\exp\left(-\frac{1}{\delta_i} \left(X_T^{i,\varphi^i} - \frac{\theta_i}{n} \sum_{j=1}^n X_T^{j,\varphi^j} \right) \right) \right],
\end{equation}
which can  easily be rewritten as 
\begin{equation}
\mathbb{E}\left[- \exp\left(-\frac{1}{\tilde{\delta}_i} \left(X_T^{i,\varphi^i} - \tilde{\theta}_i\bar{X}_T^{-i} \right) \right)\right] \label{eq lz nash problem}
\end{equation}
by introducing the parameters $\tilde{\delta}_i = \frac{\delta_i}{1-\frac{\theta_i}{n}}$ and $\tilde{\theta_i} = \frac{\theta_i}{1-\frac{\theta_i}{n}}$. Hence solving the optimization problem associated to \eqref{eq lz nash problem} yields the (one-dimensional) Nash equilibrium in terms of invested amounts
\begin{equation}
\pi^i = \left(\delta_i + \theta_i \frac{\bar{\delta}_n}{1-\bar{\theta}_n} \right)\frac{\mu}{\sigma^2},
\end{equation}
where  $\bar{\theta}_n := \frac{1}{n}\sum_{j=1}^n \theta_j$, $\bar{\delta}_n = \frac{1}{n}\sum_{j=1}^n \delta_j$.
\end{remark}

\subsection{Market with stochastic volatility}
Next we consider a stochastic volatility model. To keep the exposition simple we concentrate on the so-called Heston model where we have only one risky asset. For multivariate extensions see e.g. \cite{bauerle2013optimal}. A slightly more general stochastic volatility model with two investors, CRRA utility and ratio of competing wealth has been considered in \cite{kraft2020dynamic}. The riskless bond has again interest rate zero. There are two correlated Brownian motions $W^S$ and $W^Z$ with correlation $\rho.$ The price process $S$ of the risky asset is given by
\begin{eqnarray*}
\mathrm{d} S(t) &=&  S(t) \left( \lambda Z(t) \mathrm{d}t + \sqrt{Z(t)} \mathrm{d}W^S(t)\right),\\
\mathrm{d} Z(t) &=& \kappa (\theta -Z(t)) \mathrm{d}t + \sigma \sqrt{Z(t)} \mathrm{d} W^Z(t),
\end{eqnarray*}
where the constants $\lambda, \kappa, \theta, \sigma$ are assumed to be positive and to satisfy the Feller condition $2 \kappa \theta \ge \sigma^2$ in order to ensure that $Z$ is positive. 

We assume here that each agent uses a utility function of the form
\begin{equation*}
U_i:(0,\infty) \rightarrow \R,\ U_i(x) = \left(1-\frac{1}{\delta_i} \right)^{-1}x^{1-\frac{1}{\delta_i}},
\end{equation*}
for some parameter $\delta_i > 0$, $\delta_i \neq 1$, $i=1,\ldots,n$. The associated objective function is then given by 
\begin{equation*}
\mathbb{E}\left[\left(1-\frac{1}{\delta_i}\right)^{-1}\left(X_T^{i,\varphi^i} - \theta_i \bar{X}^{-i,\varphi}_T\right)^{1-\frac{1}{\delta_i}}\right].
\end{equation*}
The competition weights $\theta_i \in [0,1]$ are chosen with respect to the condition $x_0^i - \frac{\theta_i}{n}\sum_{j\neq i} x_0^j>0$ for the initial capitals $x_0^i>0$, $i=1,\ldots,n$.
Again, let $\tilde{x}_0^i = x_0^i - \frac{\theta_i}{n}\sum_{j\neq i} x_0^j$. The unique (up to modifications) optimal solution of the classical portfolio optimization problem 
\begin{equation}\label{eq optimization problem auxiliary2}
\begin{cases}
& \sup_{\psi^i\in \mathcal{A}} \mathbb{E}\left[\left(1-\frac{1}{\delta_i}\right)^{-1}\left(Y_T^{i,\psi^i}\right)^{1-\frac{1}{\delta_i}}\right],\\
\text{s.t.} &  Y_T^{i,\psi^i} = \tilde{x}_0^i+ (\psi^i \cdot  S)_T,
\end{cases}
\end{equation}
is given by
\begin{equation*}
\frac{\psi^{i,*}(t) S(t)}{Y_t^{i,*}} = \delta_i \lambda + f_i(t)
\end{equation*}
$\mathbb{P}$-almost surely for all $t\in [0,T]$ where the deterministic function $f_i$ can be given explicitly (Sec. 3.1 in \cite{kallsen2010utility}). Finally, $\varphi^{i,*}$, $i=1,\ldots,n$, given by 
\begin{equation*}
\varphi^{i,*}(t) = \frac{n}{n+\theta_i}\psi^{i,*}(t) + \frac{\theta_i}{(n+\theta_i)( 1-\hat{\theta})}\sum_{j=1}^n \frac{n}{n+\theta_j} \psi^{i,*}(t)
\end{equation*}
$\mathbb{P}$-almost surely for all $t\in [0,T]$, is the unique Nash equilibrium in the given situation. \medskip

Setting $\kappa, \sigma = 0$, $\lambda = \frac{\widetilde{\mu}}{\widetilde{\sigma}^2}$ and $Z(0) = \widetilde{\sigma}^2$ reduces the Heston model to the classical Black-Scholes model with constant drift $\widetilde{\mu}$ and constant volatility $\widetilde{\sigma}$. Then the optimal solution $\psi^{i,*}$ to the classical problem \eqref{eq optimization problem auxiliary2} is given by 
\begin{equation}
    \frac{\psi^{i,*}(t)S(t)}{Y_t^{i,*}} = \delta_i\frac{\widetilde{\mu}}{\widetilde{\sigma}^2}.
\end{equation}

\subsection{Cumulative Prospect Theory (CPT)}\label{subsec:cpt}
In the CPT model \cite{tversky1992advances}, investors evaluate their wealth with respect to some reference point $\xi>0$. Values smaller than $\xi$ are treated as losses while values larger than $\xi$ are seen as gains. Studies have shown that people tend to act risk-seeking when dealing with losses and risk-averse when dealing with gains. This effect is captured by S-shaped utility functions for example of the form
\begin{equation}
    U(x) = \begin{cases}
    -a\cdot (\xi-x)^{\delta},\, & x\leq \xi,\\
    b \cdot (x-\xi)^{\gamma},\, & x > \xi,
    \end{cases}
\end{equation}
for $0<\delta\leq 1,\, 0 < \gamma < 1$ and $a> b > 0$. \cite{berkelaar2004optimal} found the unique optimal solution to the associated single investor portfolio optimization problem in a Black-Scholes market with constant market parameters. The stock price processes are hence given by 
$$ \mathrm{d}S_k(t) = S_k(t)\left(\mu_k \mathrm{d}t + \sum_{\ell =1 }^d \sigma_{k\ell} \mathrm{d}W_\ell(t)\right),\, t\in [0,T],\, k=1,\ldots,d. $$

Since in our models investors evaluate their wealth with respect to the wealth of the other investors in the market, we use a reference point in terms of the weighted mean of the other investors wealth. Therefore the objective function of agent $i$ is given by 
\begin{align*}
    \EE& \left[  -a_i \cdot \left( \frac{\theta_i}{n} \sum_{j\neq i} X_T^{j,\varphi^j} - X_T^{i,\varphi^i} \right)^{\delta_i} \mathbbm{1}\left\{X_T^{i,\varphi^i} \leq    \frac{\theta_i}{n}\sum_{j\neq i} X_T^{j,\varphi^j}\right\}\right.\\
    &\quad  \left. +\,  b_i \cdot \left(X_T^{i,\varphi^i} -  \frac{\theta_i}{n}\sum_{j\neq i} X_T^{j,\varphi^j}  \right)^{\gamma_i} \mathbbm{1}\left\{X_T^{i,\varphi^i} > \frac{\theta_i}{n}\sum_{j\neq i} X_T^{j,\varphi^j} \right\}  \right]
\end{align*}
for $0<\delta_i\leq 1$, $0 < \gamma_i < 1$ and $a_i > b_i > 0$, $i=1,\ldots,n$. We further introduce the constraint $X_T^{i,\varphi^i} - \frac{\theta_i}{n}\sum_{j\neq i} X_T^{j,\varphi^j} \geq - \xi_i$ for some $\xi_i > 0$. Economically this means that the investor only accepts a downward deviation from the weighted average wealth of the other investors by a constant $\xi_i$. Hence we obtain the following optimization problem
\begin{align*}
    \begin{cases} 
         \, &\sup_{\varphi^i \in \mathcal{A}} \EE \left[  -a_i \cdot \left( \frac{\theta_i}{n} \sum_{j\neq i} X_T^{j,\varphi^j} - X_T^{i,\varphi^i} \right)^{\delta_i} \mathbbm{1}\left\{X_T^{i,\varphi^i} \leq   \frac{\theta_i}{n}\sum_{j\neq i} X_T^{j,\varphi^j}\right\}\right.\\
    &\qquad  \left. +\,  b_i \cdot \left(X_T^{i,\varphi^i} -  \frac{\theta_i}{n}\sum_{j\neq i} X_T^{j,\varphi^j}  \right)^{\gamma_i} \mathbbm{1}\left\{X_T^{i,\varphi^i} > \frac{\theta_i}{n}\sum_{j\neq i} X_T^{j,\varphi^j} \right\}  \right],\\
    \quad \text{s.t. } & X_T^{i,\varphi^i} =  x_0^i + (\varphi^i \cdot S)_T,\,  X_T^{i,\varphi^i} - \frac{\theta_i}{n}\sum_{j\neq i} X_T^{j,\varphi^j} \geq - \xi_i.
    \end{cases}    
\end{align*}

\medskip

The unique solution to the associated classical problem 
\begin{equation}\label{eq:classical_cpt}
    \begin{cases}
    & \sup_{\psi^i \in \mathcal{A}}\E\left[-a_i \left(\xi_i - Y_T^{i,\psi^i} \right)^{\delta_i} \mathbbm{1}\left\{Y_T^{i,\psi^i} \leq \xi_i \right\} + b_i \left(Y_T^{i,\psi^i} - \xi_i \right)^{\gamma_i} \mathbbm{1}\left\{Y_T^{i,\psi^i} > \xi_i \right\} \right],\\
    \text{s.t. } & Y_T^{i,\psi^i} = \widetilde{x}_0^i + \xi_i + (\psi^i \cdot S)_T, \, Y_T^{i,\psi^i} \geq 0,
    \end{cases}
\end{equation}
is then given  by \cite{berkelaar2004optimal} 
\begin{align*}
    \psi_k^{i,*}(t)S_k(t) =  &\Big(\Big(\sigma \sigma^\top \Big)^{-1} \mu\Big)_k \cdot \left\{ \frac{\xi_i \phi\left(g(t,\bar{L}_i) \right)}{\Vert \kappa \Vert \sqrt{T-t}}\right. \\
    & \left.+ \left(\frac{b_i \gamma_i}{\lambda_i L(t)}\right)^{\frac{1}{1-\gamma_i}} \mathrm{e}^{\Gamma_i(t)} \left(\frac{\phi\left(g(t,\bar{L}_i) + \frac{\Vert \kappa \Vert \sqrt{T-t}}{1-\gamma_i} \right)}{\Vert \kappa \Vert \sqrt{T-t}} + \frac{\Phi\left(g(t,\bar{L}_i) + \frac{\Vert \kappa \Vert \sqrt{T-t}}{1-\gamma_i} \right)}{1-\gamma_i} \right)  \right\},
\end{align*}    
where $\kappa$ and $L$ are the market price of risk and state price density in the given Black-Scholes market and $\phi$ and $\Phi$ describe the density and cumulative distribution function of the standard normal distribution. Moreover, $g$ and $\Gamma_i$ are functions in terms of the market parameters, whereby $g$ also depends on $L$. Finally, $\lambda_i$ is the Lagrange multiplier to the constraint $\EE[L(T)Y_T^{i,\psi^i}]=\widetilde{x}_0^i+\xi_i$ and $\bar{L}_i$ is the unique root of some additional function $f_i$ depending only on the market and preference parameters of the problem. Explicit representations of the mentioned functions can be found in \cite{berkelaar2004optimal}. Using those explicit representations, it can be easily seen that $\psi_k^{i,*}(t)>0$, which implies, together with Remark \ref{remark:mon}, that the associated component $\varphi^{i,*}_k(t)$ of the Nash equilibrium  (see below) is increasing in terms of $\theta_i$.\\ 
Finally, $\varphi_k^{i,*}$, $i=1,\ldots,n$, $k=1,\ldots,d$ given by 
\begin{equation*}
\varphi_k^{i,*}(t) = \frac{n}{n+\theta_i}\psi_k^{i,*}(t) + \frac{\theta_i}{(n+\theta_i)( 1-\hat{\theta})}\sum_{j=1}^n \frac{n}{n+\theta_j} \psi_k^{i,*}(t)
\end{equation*}
$\mathbb{P}$-almost surely for all $t\in [0,T]$, is the unique Nash equilibrium in the given situation. \medskip

Taking the limit $\xi_i \to 0$ in \eqref{eq:classical_cpt} yields a classical utility maximization problem with power utility $U_i(y) = b_i y^{\gamma_i}$, $y>0$. As $\xi_i \to 0$, the associated optimal solution $\psi_k^{i,*}$ then converges to 
\begin{equation}
    \psi_k^{i,*}(t) S_k(t) = \frac{1}{1-\gamma_i}\Big( \Big(\sigma \sigma^\top\Big)^{-1}\mu\Big)_k \left(\frac{b_i\gamma_i}{\lambda_i L(t)} \right)^{\frac{1}{1-\gamma_i}}\mathrm{e}^{\Gamma_i(t)}.
\end{equation}

Moreover, if we set $b_i = \frac{1}{\gamma_i}$ we obtain the well-known Merton ratio
\begin{equation}
    \frac{\psi_k^{i,*}(t)S_k(t)}{Y_t^{i,*}}=\frac{1}{1-\gamma_i}\Big(\Big(\sigma\sigma^\top\Big)^{-1}\mu\Big)_k.
\end{equation}

\subsection{ Cox-Ross-Rubinstein model}\label{subsec:CRR}
Our last example is a Cox-Ross-Rubinstein market in discrete time. As already stated in Remark \ref{remark:discrete_vs_continuous}, all arguments carry over for problems in discrete time. 
The market consists of one riskless bond with price process 
\begin{equation}
     B(t_k) = 1,\quad k = 0,\ldots, N,
\end{equation}
 and one stock with price process $(S(t_k))_{k=0,\ldots,N}$, where $t_k = k\cdot \frac{T}{N}$, $k=0,\ldots,N$. The stock price process is given by 
\begin{equation}
    S(t_k) = S(0) \prod_{\ell=1}^{k} {R}_{\ell},\ k=0,\ldots,N.
\end{equation}
The random variables ${R}_k$, $k=1,\ldots,N$, are i.i.d. with $\PP({R}_{k} = u) = p = 1-\PP({R}_{k} = d)$ for $0<d<u$ and $p\in (0,1).$ We assume that $d < 1 < u$ to exclude arbitrage \cite{bauerle2011markov,bauerle2017finanzmathematik}.  Note that the wealth process satisfies 
$$ X^\varphi(t_k) = x_0+\sum_{\ell=1}^{k} \varphi(t_{\ell-1})  (S(t_\ell)-S(t_{\ell-1}))$$
and is thus again linear in $\varphi.$

We assume that investors use exponential utility functions given by
$$U_i:\R \to \R,\, U_i(x) = -\exp\left(-\frac{1}{\delta_i}x\right),\, i=1,\ldots,n. $$
Using \cite[p.~119]{bauerle2017finanzmathematik} the unique optimal solution to the classical optimization problem
\begin{equation}
   \begin{cases}
   & \sup_{\psi^i\in \mathcal{A}} \EE\left[-\exp\left(-\frac{1}{\delta_i} Y^{i,\psi^i}(T) \right) \right],\\
   \text{s.t.} & Y^{i,\psi^i}(T) = \tilde{x}_0^i + (\psi^i \cdot S)_T,
   \end{cases}
\end{equation}
is given by
\begin{equation}
\psi^{i,*}(t_k) = \frac{ \delta_i}{S(t_k)} \frac{\log\left(\frac{1-q}{1-p}\right)-\log\left(\frac{q}{p}\right)}{u-d},
\end{equation}
where $q = \frac{1-d}{u-d}.$

Then the described method yields the unique Nash equilibrium $\pi^{i,*}$, $i=1,\ldots,n$, where
\begin{equation}
\pi^*_i = \left(\frac{n}{n+\theta_i}\delta_i + \frac{\theta_i}{(n+\theta_i)(1-\hat{\theta})}\sum_{j=1}^n \frac{n}{n+\theta_j}\delta_j\right) \frac{\log\left(\frac{1-q}{1-p}\right)-\log\left(\frac{q}{p}\right)}{u-d},\label{eq NE CRR exp}
\end{equation}
is the invested amount. 
Similar to the Nash equilibrium in Remark \ref{remark:no_jumps}, the invested amount is constant in time and given as the constant factor $C_i$ times some expression depending only on the market parameters. If the expression $(u-d)^{-1}\left(\log\left(\frac{1-q}{1-p}\right)-\log\left(\frac{q}{p}\right)\right)$ is strictly positive (this is equivalent to $p>q$) we can use the same argumentation regarding the monotonicity of $C_i$ as before.

\begin{remark}
The Cox-Ross-Rubinstein model brings the advantage of being very simple, but also quite popular among financial markets in discrete time. Moreover, the Cox-Ross-Rubinstein model is a discrete time approximation of the Black-Scholes market which we already saw as a special case of the Lévy market in Subsection \ref{subsec:levy}. Hence, it does not come as a surprise that the overall structures of the Nash equilibria in the Black-Scholes market with constant parameters and the Cox-Ross-Rubinstein market coincide.
\end{remark}

\section{The mean field Problem}\label{sec:MF}

In this section we consider the limit of the previous $n$-agent game as $n$ tends to infinity. Hence we work with a continuum of agents and study the optimal investment of some representative investor whose initial capital and preference parameters are realizations of suitable random variables. \medskip

To motivate the subsequent definition, we provide an informal derivation of the limit of the Nash equilibrium \eqref{eq nash equilibrium} as $n\to \infty$. Obviously, we obtain 
$$ \lim_{n\to \infty} \frac{n}{n+\theta_i} =1,\quad \lim_{n\to \infty} \frac{n\theta_i}{n+\theta_i} = \theta_i. $$
Moreover, if we assume that $\theta_1,\theta_2,\ldots$ are i.i.d. random variables, we obtain
$$ \hat{\theta} = \sum_{j=1}^n \frac{\theta_j}{n+\theta_j} \stackrel{\text{a.s.}}{\longrightarrow} \E[\theta_1],\, n \to \infty, $$
since, using the law of large numbers, 
\begin{align*} 
&\hat{\theta} = \sum_{j=1}^n \frac{\theta_j}{n} \cdot \frac{n}{n+\theta_j} \leq \sum_{j=1}^n \frac{\theta_j}{n} \stackrel{\text{a.s.}}{\longrightarrow} \E[\theta_1],\, n \to \infty,\\
& \hat{\theta} = \sum_{j=1}^n \frac{\theta_j}{n} \cdot \frac{n}{n+\theta_j} \geq \frac{n}{n+1} \sum_{j=1}^n \frac{\theta_j}{n} \stackrel{\text{a.s.}}{\longrightarrow} \E[\theta_1],\, n \to \infty.
\end{align*}

Finally, if we assume that, conditional on $\mathcal{F}_T$, $\psi_k^{1,*}(t),\psi_k^{2,*}(t),\ldots$ are i.i.d. random variables (for any $t\in [0,T]$, $k\in \{1,\ldots,d\}$) we obtain (analogously, using the law of large numbers and a sandwich argument)
$$ \sum_{j=1}^n \frac{1}{n+\theta_j} \psi^{j,*}_k(t) \stackrel{\text{a.s.}}{\longrightarrow} \E\big[\psi^{1,*}_k(t)|\mathcal{F}_T\big],\, n \to \infty. $$
Note that we can only assume that $\psi^{j,*}_k(t)$ are i.i.d. \textit{given $\mathcal{F}_T$} as they are solutions to portfolio optimization problems at time $T$.\medskip

Hence, we expect that the components $\varphi_k^{i,*}(t)$ of the Nash equilibrium \eqref{eq nash equilibrium} converge to
\begin{equation}
    \psi_k^{*}(t) + \frac{\theta}{1-\E[\theta]} \E[\psi_k^{*}(t)|\mathcal{F}_T], \quad k=1,\ldots,d
\end{equation}
as $n\to \infty$, if $\theta \sim \theta_1$ and $\psi_k^*(t)\sim \psi_k^{1,*}(t)$ (given $\mathcal{F}_T$). \medskip 

Now we give a formal definition of a mean-field equilibrium and prove that it does indeed coincide with the informally derived limit above. We will again use the financial market described in Section \ref{sec:market}. Moreover, we assume that the underlying probability space additionally contains a $(0,\infty) \times (0,\Delta] \times [0,1]$-valued random vector $\zeta = (\xi, \delta, \theta)$ independent of the filtration $(\mathcal{F}_t)$ with $\Delta>0$. We assume that $ \EE[\xi^2]<\infty.$ Finally, we define an additional filtration $(\mathcal{F}_t^{\text{MF}})_{t\in [0,T]}$ given by 
\begin{equation*}
\mathcal{F}^{\text{MF}}_t \coloneqq \sigma\left(\mathcal{F}_t,\ \zeta\right).
\end{equation*}

The random variables $\xi$ and $\delta,\ \theta$ denote the initial capital and preference parameters of some representative investor. \medskip

In this setting the wealth process of some representative investor is  given by 
\begin{equation}
    X_t^\varphi = \xi + (\varphi \cdot S)_t,\ t\in [0,T],
\end{equation}
where $\varphi$ is an admissible strategy representing the number of stocks held at time $t\in [0,T]$. We say that $\varphi$ is an admissible strategy if $\varphi \in \mathcal{A}^{\text{MF}}$, where
\begin{align*}
\AMF \coloneqq \Big\{ &\varphi \in L^{\text{MF}}(S): \,  (\varphi\cdot S)_T \in L^2(\mathbb{P}), \; (\varphi \cdot S) Z^\mathbb{Q} \text{ is a $\mathbb{P}$-martingale for all S}\sigma\text{MM } \mathbb{Q} \\
 & \text{ with density process } Z^\mathbb{Q} \text{ and  } \frac{\dx \mathbb{Q}}{\dx \mathbb{P}} \in L^2(\mathbb{P})\Big\}.
\end{align*}
By $L^{\text{MF}}(S)$ we denote the set of $(\mathcal{F}_t^{\text{MF}})$-predictable, $S$-integrable stochastic processes. \medskip

Assume that $U:\mathcal{D}\rightarrow \R$ is a utility function defined on a domain $\mathcal{D}\in \{\R,\, [0,\infty),\, (0,\infty)\}$ including some parameter $\delta>0$. We again extend the definition of $U$ on $\R$ by setting $U(x)=-\infty$ if $x\notin \mathcal{D}$. Now assume that $\delta$ is part of the characterization of the representative investor. Then the combination of $U$ and $\delta$ (by inserting $\delta$ as the parameter in $U$) yields a (stochastic) utility function denoted by $U_\delta$.

Then a representative investor faces the following optimization problem \smallskip
\begin{equation}
\begin{cases} & \sup_{\varphi \in \mathcal{A}^{\text{MF}}} \EE\left[U_\delta \left(X_T^\varphi- \theta \bar{X}\right)\right],\\
\text{s.t. } & X_T^\varphi = \xi + (\varphi \cdot S)_T,\, \bar{X} = \EE[X_T^\varphi| \mathcal{F}_T].
\end{cases} \label{eq problem mean field general}
\end{equation}
 
\begin{remark}
We need to ensure that there is at least one strategy s.t. $X_T^\varphi- \theta \bar{X} \in \mathcal{D}$ $\mathbb{P}$-almost surely.
 If $\mathcal{D}\in \{[0,\infty), (0,\infty)\}$ we can therefore only allow choices of $\xi$ and $\theta$ that satisfy $\xi - \theta \bar{\xi}\geq 0 (>0)$ $\PP$-almost surely, where $\bar{\xi} \coloneqq \EE[\xi]$.
\end{remark}

The optimal solution to \eqref{eq problem mean field general} is called \emph{mean field equilibrium}.

\begin{definition}
An admissible strategy $\varphi^*\in \mathcal{A}^{\text{MF}}$ is called a \emph{mean field equilibrium} (MFE), if it is an optimal solution to the optimization problem \eqref{eq problem mean field general}. This means in particular that $\varphi^*$ needs to satisfy $\bar{X} = \E[X_T^{\varphi^*}| \mathcal{F}_T]$.
\end{definition}

The optimization problem \eqref{eq problem mean field general} can be solved similarly to the $n$-agent problem. Therefore we define the following auxiliary problem 
\begin{equation}
\begin{cases}
 & \sup_{\psi \in \AMF} \EE\left[U_\delta(Z^\psi_T)\right],\\
\st & Z^\psi_T = \xi - \theta \bar{\xi} + (\psi \cdot S)_T.
\end{cases}\label{eq aux problem mean field general}
\end{equation}

Then the mean field equilibrium to \eqref{eq problem mean field general} is given by the following theorem in terms of the optimal solution to the auxiliary problem \eqref{eq aux problem mean field general}.

\begin{theorem}\label{thm mean field general}
Let $\E[\theta]\eqqcolon \bar{\theta} < 1$. If \eqref{eq aux problem mean field general} has a unique (up to modifications) optimal portfolio strategy $\psi^{*}$, then there exists a unique mean field equilibrium for \eqref{eq problem mean field general} given by 
\begin{equation}\label{eq:NE_MF}
\varphi_k^*(t) = \psi_k^*(t) + \frac{\theta}{1-\bar{\theta}} \EE\left[\psi_k^*(t) | \mathcal{F}_T\right],\, k=1,\ldots,d,
\end{equation}
$\PP$-almost surely for all $t\in [0,T]$.
\end{theorem}

\begin{proof}
In order to solve the optimization problem \eqref{eq problem mean field general}, we assume that $\bar{X}$ is some fixed $\mathcal{F}_T$-measurable random variable of the form 
\begin{equation*}
\bar{X} = \EE\left[X_T^\alpha | \mathcal{F}_T\right]
\end{equation*} 
for some admissible strategy $\alpha \in \AMF$ with $X_0^\alpha = \xi$. Moreover, we define the random variable 
\begin{equation}
\bar{X}_t^\alpha \coloneqq \EE\left[X_t^\alpha  | \mathcal{F}_T\right], \ t\in [0,T].
\end{equation}
Since $(X_t^\alpha)$ can be written as 
\begin{equation}
    X_t^\alpha = \xi + \sum_{k=1}^d \int_0^t \alpha_k(u)\dx S_k(u),
\end{equation}

we obtain 
\begin{equation}\label{eq:X_alpha_bar}
\bar{X}_t^\alpha = \bar{\xi} + \sum_{k=1}^d \int_0^t \bar{\alpha}_k(u) \dx S_k(u),
\end{equation}
where $\bar{\alpha}_k(u) \coloneqq \EE[\alpha_k(u)| \mathcal{F}_T]$, $k=1,\ldots,d,$ and $\bar{\xi} = \EE[\xi]$. \medskip

The representation \eqref{eq:X_alpha_bar} of $\bar{X}_t^\alpha$ requires a little more explanation. At first, we used the independence of $\xi$ and $(\mathcal{F}_t)$ and the linearity of the conditional expectation. To obtain 
\begin{equation}\label{eq:switch}
    \EE\left[\int_0^t \alpha_k(u) \dx S_k(u) \Big| \mathcal{F}_T \right] = \int_0^t \bar{\alpha}_k(u) \dx S_k(u),
\end{equation}
we first observe that the sample space can be written as $\Omega = \Omega_1 \times \Omega_2$ with $\sigma$-algebras $\mathcal{F}^{(1)}$ and $\mathcal{F}^{(2)}$ on $\Omega_1,\, \Omega_2$, respectively, and probability measures $\PP_1$ and $\PP_2$ on $\mathcal{F}^{(1)},\, \mathcal{F}^{(2)}$, satisfying $\PP = \PP_1 \otimes \PP_2$ due to the independence of $(\mathcal{F}_t)$ and $\zeta$. Hence $\Omega_1$ and $\Omega_2$ are associated to $(\mathcal{F}_t)$ and $\zeta$, respectively. It follows that for any $F\in\mathcal{F}_T$ there exists a unique $G\in \mathcal{F}^{(1)}$ such that $F = G \times \Omega_2$.

Now we can show that the conditional expectation of some $\mathcal{F}_T^{\text{MF}}$-measurable, $\PP$-integrable random variable $Y = Y(\omega_1, \omega_2)$ can be written as an integral over $\Omega_2$ with respect to $\PP_2$. At first we notice that $\EE[Y|\mathcal{F}_T]$ is a random variable that can be written in terms of $\omega_1$ only, due to the $\mathcal{F}_T$-measurability. Now let $F\in \mathcal{F}_T$ with decomposition $F=G\times \Omega_2.$ It follows (using Fubini's theorem and the definition of conditional expectation)
\begin{align*}
    \int_G \int_{\Omega_2} Y(\omega_1, \omega_2) \dx \PP_2(\omega_2) \dx \PP_1(\omega_1) &= \int_F Y(\omega_1, \omega_2) \dx \PP(\omega_1, \omega_2) \\
    &= \int_F \EE[Y|\mathcal{F}_T](\omega_1, \omega_2) \dx \PP(\omega_1, \omega_2) \\
    &= \int_G \int_{\Omega_2} \EE[Y|\mathcal{F}_T](\omega_1, \omega_2) \dx \PP_2(\omega_2) \dx \PP_1(\omega_1)\\
    &= \int_G \int_{\Omega_2} \EE[Y|\mathcal{F}_T](\omega_1) \dx \PP_2(\omega_2) \dx \PP_1(\omega_1)\\
    &=\int_G \EE[Y|\mathcal{F}_T](\omega_1) \dx \PP_1(\omega_1).
\end{align*}
Therefore we obtain 
\begin{equation}
    \EE[Y|\mathcal{F}_T](\omega_1) = \int_{\Omega_2} Y(\omega_1, \omega_2) \dx \PP_2(\omega_2)
\end{equation}
$\PP_1$-almost surely. Now we can use this result to prove \eqref{eq:switch}: 
\begin{align*}
     \EE\left[\int_0^t \alpha_k(u) \dx S_k(u) \Big| \mathcal{F}_T \right] &= \int_{\Omega_2} \int_0^t \alpha_k(u) \dx S_k(u) \dx \PP_2(\omega_2) \\
     &= \int_0^t \int_{\Omega_2} \alpha_k(u) \dx\PP_2(\omega_2) \dx S_k(u) \\
     &= \int_0^t \EE[\alpha_k(u)|\mathcal{F}_T] \dx S_k(u) \\
     &= \int_0^t \bar{\alpha}_k(u) \dx S_k(u).
\end{align*}
The second equality holds due to a stochastic version of Fubini's theorem that can be found in \cite{protter2005}, Theorem 65. Hence the representation \eqref{eq:X_alpha_bar} is in fact correct and we can proceed with the solution of problem \eqref{eq problem mean field general}.\medskip

By construction, the equation $\bar{X}_T^\alpha = \bar{X}$ is satisfied.
Using the previously defined process $(\bar{X}_t^\alpha)$, we define another process $(Z_t)_{t\in [0,T]}$ for $\varphi \in \mathcal{A}^{MF}$ and $X_0^\varphi = \xi$ by 
\begin{equation*}
Z_t \coloneqq X_t^\varphi - \theta \bar{X}_t^\alpha.
\end{equation*}
Thus $(Z_t)$ can be written as
\begin{equation*}
Z_t = \xi - \theta \bar{\xi} + \sum_{k=1}^d \int_0^t \underbrace{(\varphi_k(u) - \theta \bar{\alpha}_k(u))}_{\eqqcolon \psi(u)} \dx S_k(u) \eqqcolon Z_t^\psi.
\end{equation*}

At maturity $T$, the random variable $Z^\psi_T$ coincides with the argument of the objective function in \eqref{eq problem mean field general}. Therefore we consider the auxiliary problem \eqref{eq aux problem mean field general}
\begin{equation*}
\begin{cases}
 & \sup_{\psi \in \AMF} \EE\left[U_\delta(Z^\psi_T)\right],\\
\st & Z_T^\psi = \xi - \theta \bar{\xi} + (\psi \cdot S)_T. 
\end{cases}
\end{equation*} 
If $\psi^*$ is the optimal portfolio strategy to \eqref{eq aux problem mean field general}, we can determine the solution to \eqref{eq problem mean field general} as follows. By definition we have $Z_T^{\psi^*} = X_T^\varphi - \theta \bar{X}$ or equivalently $X_T^\varphi = Z_T^{\psi^*} + \theta \bar{X}$. Moreover, the random variable $\bar{X}$ needs to satisfy the constraint 
\begin{equation*}
\bar{X} = \EE\left[X_T^\varphi | \mathcal{F}_T\right].
\end{equation*}
Hence it follows 
\begin{align*}
\bar{X} & = \EE\left[X_T^\varphi | \mathcal{F}_T\right] = \EE\left[Z^{\psi^*}_T + \theta \bar{X} \big| \mathcal{F}_T\right]= \EE\left[Z^{\psi^*}_T\big|\mathcal{F}_T\right] + \bar{X} \EE\left[\theta | \mathcal{F}_T\right]= \EE\left[Z^{\psi^*}_T\big|\mathcal{F}_T\right] + \bar{\theta} \bar{X},
\end{align*}
where we used that $\bar{X}$ is $\mathcal{F}_T$-measurable and that $\theta$ is independent of $(\mathcal{F}_t)$. Moreover, we introduced the notation $\bar{\theta} \coloneqq \EE[\theta]$. Under the assumption that $\bar{\theta} < 1$, we obtain 
\begin{equation*}
\bar{X} = \frac{1}{1-\bar{\theta}} \EE\left[Z^{\psi^*}_T\big|\mathcal{F}_T\right].
\end{equation*}
Therefore the optimal wealth $X_T^\varphi$ is given by 
\begin{equation*}
X_T^\varphi = Z^{\psi^*}_T + \theta \bar{X} = Z^{\psi^*}_T + \frac{\theta}{1-\bar{\theta}} \EE\left[Z^{\psi^*}_T|\mathcal{F}_T\right].
\end{equation*}
Since the wealth process is linear in terms of the strategy, it follows that 
\begin{equation*}
\varphi(t) = {\psi^*}(t) + \frac{\theta}{1-\bar{\theta}} \EE\left[\psi^*(t) | \mathcal{F}_T\right]
\end{equation*}
componentwise $\PP$-almost surely for all $t\in [0,T]$. Again, the line of arguments implies that there exists a unique Nash equilibrium given by \eqref{eq:NE_MF} if and only if the auxiliary problem \eqref{eq aux problem mean field general} is uniquely solvable.
\end{proof}
\begin{remark}
The Nash equilibrium \eqref{eq:NE_MF} shows, similar to Remark \ref{remark:mon}, that a larger value of $\theta$ results in a more risky investment behavior of some representative agent. If we substitute $\theta$ by a different $[0,1]$-valued random variable $\tilde{\theta}$ with $\EE[\tilde{\theta}]<1$ and $\tilde{\theta}>\theta$ $\PP$-almost surely, the resulting Nash equilibrium becomes more risky in the sense that more shares of the risky asset are purchased or sold short depending on whether the realization of $\psi_k^*(t)$ is positive or negative.
\end{remark}

\begin{example} We consider a one-dimensional Black-Scholes financial market with constant drift $\mu>0$ and volatility $\sigma>0$. Moreover, let $U_\delta(x) = -\exp\left(-\frac{1}{\delta}x\right)$, $x\in \R$. Then the solution to the auxiliary problem \eqref{eq aux problem mean field general} in terms of amounts is known to be given by 
\begin{equation*}
\pi^Z = \delta \frac{\mu}{\sigma^2}.
\end{equation*}
Therefore the mean field equilibrium to \eqref{eq problem mean field general} in terms of amounts  is given by 
\begin{equation*}
\pi = \delta \frac{\mu}{\sigma^2} + \frac{\theta}{1-\bar{\theta}} \EE \left[\delta \frac{\mu}{\sigma^2} \Big| \mathcal{F}_T^S\right] = \left(\delta + \bar{\delta} \frac{\theta}{1-\bar{\theta}}\right) \frac{\mu}{\sigma^2}.
\end{equation*}
\end{example}

\bibliographystyle{amsplain}
\bibliography{literatur}

\end{document}